\documentclass[]{amsart}

\newtheorem{definition}{Definition}[section]

\newtheorem{remark}[definition]{Remark}

\newtheorem{lemma}[definition]{Lemma}
\newtheorem{proposition}[definition]{Proposition}
\newtheorem{theorem}[definition]{Theorem}

\newcommand{\K}{\mathbb{K}}

\newcommand{\N}{\mathbb{N}}

\newcommand{\cO}{\mathcal{O}}

\AtBeginDocument{%
  \providecommand\BibTeX{{%
    \normalfont B\kern-0.5em{\scshape i\kern-0.25em b}\kern-0.8em\TeX}}}

\begin{document}

\title[Bounds for degrees of syzygies]{Bounds for degrees of syzygies of polynomials defining a grade two ideal}

\author{Teresa Cortadellas Ben\'itez}
\address{Universitat de Barcelona, Facultat d'Educaci\'o.
Passeig de la Vall d'Hebron 171,
08035 Barcelona, Spain}
\email{terecortadellas@ub.edu}

\author{Carlos D'Andrea}
\address{Universitat de Barcelona, Departament de Matem\`atiques i Inform\`atica,
 Universitat de Barcelona (UB),
 Gran Via de les Corts Catalanes 585,
 08007 Barcelona,
 Spain} \email{cdandrea@ub.edu}
\urladdr{http://www.ub.edu/arcades/cdandrea.html}

\author{Eul\`alia Montoro}
\address{Universitat de Barcelona, Departament de Matem\`atiques i Inform\`atica,
 Universitat de Barcelona (UB),
 Gran Via de les Corts Catalanes 585,
 08007 Barcelona,
 Spain}
\email{eula.montoro@ub.edu}



\begin{abstract}
We make explicit the exponential bound on the degrees of the polynomials appearing in the Effective Quillen-Suslin Theorem, and apply it jointly with the Hilbert-Burch Theorem to show that the syzygy module of a sequence of $m$ polynomials in $n$ variables defining a complete intersection ideal of grade two is free, and that a basis of it can be computed with bounded degrees. In the known cases, these bounds improve previous results.
\end{abstract}

\maketitle

\section{Introduction}
Let $n$ and $m$ be positive integers, $m\geq2,$ and  $R:=\K[x_1,\ldots, x_n]$ a polynomial ring in $n$ variables with coefficients in an infinite field $\K.$ Suppose that $a_1,\ldots, a_m,\,p,q\in R$ are such that the following equality of ideals in $R$ holds:
\begin{equation}\label{1}
I:=\langle a_1, a_2, \ldots, a_m\rangle = \langle p, q\rangle.
\end{equation}
In this paper, we will show  that $\mbox{Syz}(a_1,\ldots,a_m)$, the $R$-syzygy module of the sequence $(a_1,\ldots, a_m)$ is a free $R$-module, and we will develop algorithms for a computation of a basis of it in terms of the matrices of converting the $a_i$'s into $p,q$ and vice versa. Our approach will be the use of the Effective Quillen-Suslin Theorem presented in \cite{CCDHKS93}. We will also develop a bound on the degrees of the elements of such a basis as a function of the degrees of the input data.

Note that we are not assuming in principle that neither $p$ and $q$ are coprime, nor that the ideal they define equals the whole ring $R$. But our claim can be simplified in the first case by removing common factors without changing the syzygy module, and in the second case the problem is solved completely by applying the Effective Quillen-Suslin Theorem (cf. \cite{CCDHKS93}) directly to the input data, with better bounds than those presented below.  So we can assume w.l.o.g. for the rest of the text that $\gcd(p,q)=1,$ and that they do not generate the unitary ideal, i.e. that the grade of $I$ is equal to two. Let $\delta_a\in\N$ be a bound on the total degrees of $a_1,\ldots, a_m$, and $\delta_0$ a bound for the degrees of $p$ and $q$. The main result of this paper is the following.
\begin{theorem}\label{mtt}
Given the data $a_1,\ldots, a_m, p, q\in\K[x_1,\ldots, x_n]$ satisfying \eqref{1}, and $\delta_a,\,\delta_0$ as defined above. There is an algorithm which computes an $R$-basis of $\mbox{Syz}(a_1,\ldots, a_m)$ made by vectors of polynomials of degree bounded by $3n^24^n(\delta_0^2+\delta_a+1)^{2n}.$

If in addition $I$ is zero dimensional (for instance when $n=2$), then another basis of the same module can be computed with the following degree bound:
\begin{equation}\label{sb}
2\delta_0+2\delta_a+2\delta_a^2+\delta_0^2+3mn^24^n(2\delta_a^2+\delta_a+\delta_0+1)^{2n}.
\end{equation}
\end{theorem}
The proof of this Theorem is given in Section \ref{4.2}. These results can be applied to the case treated in our ISSAC 2020 paper  \cite{CDM20}, where $n=2,\,m=4$ and $p, q$ a ``Shape Lemma'' representation of the radical ideal $I$.
Our algorithm and main result there state that a basis of $\mbox{Syz}(a_1, a_2, a_3, a_4)$ can be found with degree bounded by $5\delta_0^4(2\delta_a+1)^4.$

Note that in this case, $\delta_0$ can be expressed in terms of $\delta_a$ thanks to B\'ezout Theorem, and we can set $\delta_0:=\delta_a^2.$
The results in \cite{CDM20} amount then to a bound of size $\delta_a^{12}.$ In contrast, the first bound in Theorem \ref{mtt} amounts to a constant times $\delta_a^{16}$, while the second one is of the order of $\delta_a^8,$ which is an improvement with respect to this previous bound. In addition, we will see in Proposition \ref{citam} that a careful analysis of this situation can actually make the bound of size $\delta_a^{12}$ get smaller, comparable to the results in \cite{CDM20}.

It should be pointed out, however, that even though the situation presented here contains the problem tackled in \cite{CDM20}, this paper is not a generalization of the results given there, as our methods  are slightly different  despite the fact that in both cases we use Hilbert-Burch and Effective Quillen-Suslin theorems.  Our ISSAC 2020 paper dealt with the case when $p$ and $q$ are ``shape'' basis of an ideal of $4$ polynomials in $2$ variables, and the results were restricted to that case. Here, we deal with any number of polynomials and variables, and even in the case $n=2$ we are not assuming that the ideal has a shape basis, just that it is a complete intersection of grade $2$.  In Section \ref{final} we will compare both approaches.

 The computation of syzygies of sequences of polynomials is of  major interest in the Computer Aided Geometric Design community. In the cases of curves ($n=1$) and surfaces ($n=2$), these are called ``$\mu$-bases'',  see for instance \cite{cox03, CW03, CCL05,DCS05, HHK17,SL17, YFJS19, YJ19} and the references therein.  The situation for curves is quite well-understood and classical, see for instance \cite{CSC98}. The existence of $\mu$-bases for surfaces has been proven in \cite{CCL05}, and several methods have been proposed for computing them in taylored situations (revolution, canal, translational,... surfaces). In \cite{DCS05}, a concrete method for computing a $\mu$-basis is proposed based on the Matrix Primitive Factorization Theorem given in \cite{GB82}, but no concrete bounds on the outcome are given.

 Even though the study of syzygies of ideals of grade $2$ does not cover all the cases of interest in the literature ---for instance, it is known that if $n=2$ the syzygy module of $a_1,\ldots, a_m$ is always free independently of the fact that $I$ can be generated by $2$ polynomials, see \cite{cid19} for general bounds for degrees in this case--- the situation presented in \eqref{1} is quite general in the sense that to have $\mbox{Syz}(a_1,\ldots, a_m)$ being a free module,  if all the $a_i$'s are coprime, then Hilbert-Burch Theorem \ref{hb} implies that this ideal must have grade $2,$ and hence they should be described ---at least locally--- with two polynomials.

The paper is organized as follows: in Section \ref{2} we review both Hilbert-Burch and Effective Quillen-Suslin Theorems (Theorems \ref{hb} and \ref{eqs} respectively), and give an explicit bound in Theorem \ref{explicitqs} for the degree of a unimodular ``inverse'' matrix to a unimodular one. In Section \ref{3} we show that the matrix of polynomials converting $(a_1 \ldots a_m)$ into $(p\ q)$ is unimodular, while the one reversing this process is ``almost'' unimodular, see Proposition \ref{Nu}, but can be replaced by a unimodular one (cf. Proposition \ref{Mun}). This result allows a complete characterization of those data $a_1,\ldots, a_m, p, q\in R$ satisfying \eqref{1} in terms of a unimodular conversion matrix, see Theorem \ref{char}.

All these results are then used in Section \ref{alg} to develop  algorithms which compute an $R$-basis of $\mbox{Syz}(a_1,\ldots, a_m)$ based each of them in one of these conversion matrices.

In Section \ref{4} we study degree bounds for the output of the algorithms presented, and also prove Theorem \ref{mtt}. In section \ref{5} we show some examples illustrating our methods and tools. We conclude the paper by comparing our approach with the one presented in \cite{CDM20} in Section \ref{final}.

\smallskip
\noindent\underline{\bf Acknowledgements:} All our computations were done with the aid of the softwares Maple \cite{map20} and Mathematica \cite{math18}. We also acknowledge useful conversations with Mart\'in Sombra while working some of the results of this paper. T. Cortadellas was supported by the Spanish MICINN Research projects MTM2013-40775-P and PID 2019-104844GB-100.  C. D'Andrea and E. Montoro were supported by the Spanish MICINN research projects MTM 2015-65361-P and PID2019-104047GB-I00.

\bigskip
\section{Hilbert-Burch and Effective Quillen-Suslin Theorems}\label{2}
We start by recalling the well-known Hilbert-Burch Theorem for resolutions of length $1.$
\begin{theorem}\label{hb} \cite[Theorem 3.2]{eis05}
Suppose that an ideal $I$ in a Noetherian commutative ring $A$ admits a free resolution of length $1$ as follows:
$$0\to F_1\stackrel{G}{\to}F_2\to I\to0.$$
If the rank of the free module $F_1$ is $\ell,$ then the rank of $F_2$ is $\ell+1,$ and there exists a nonzero divisor $a\in A$ such that $I$ is equal to $a$ times the ideal of $\ell\times \ell$ minors of the matrix $G$ with respect with any given bases of $F_1$ and $F_2$. The generator of $I$ that is the image of the $i$-th basis vector of $F_2$ is $\pm a$ times the determinant of the submatrix of $G$ formed from all except the $i$-th row. Moreover, the grade of the ideal of maximal minors is $2$.

Conversely, given an $(\ell+1)\times \ell$ matrix $G$ with entries in $A$ such that the grade of the ideal of $\ell\times \ell$ minors of $G$ is at least $2,$ and a given nonzero divisor $a\in A$, the ideal $I$ generated by $a$ times the $\ell\times \ell$ minors of $G$ admits a free resolution of length one as above. It has grade $2$ if and only if $a$ is a unit.
\end{theorem}
We bring also to the picture the main tool we will use in our paper, namely the Effective Quillen-Suslin Theorem.
We recall that  $R=\K[x_1,\ldots, x_n]$ is a polynomial ring in $n$ variables with coefficients in an infinite field $\K.$
A matrix $U\in R^{r\times s}$ is called \emph{unimodular} if the ideal generated by the maximal minors of it equals to the whole  ring $R.$ We denote by ${\bf I}_r$ the identity matrix of size $r\times r$. An elementary matrix is one that consists in exchanging two rows (or two columns) of ${\bf I}_r$, or adding to a row (or column) a polynomial multiple of another. The degree of a matrix equals the maximum of the degrees of its entries.
\begin{theorem}\cite[Theorem 3.1]{CCDHKS93}\label{eqs}
Assume that $F\in R^{r\times s}$ is unimodular, with $r\leq s.$ Then, there exists a square matrix $U\in R^{s\times s}$ such that
\begin{enumerate}
\item $U$ is unimodular,
\item $F\cdot U=[{\bf I}_r,{\bf 0}]\in R^{r\times s},$
\item $\deg(U)=(rd)^{\cO(n)},$ and
\item $U$ is a product of $\cO(n^2s^2(rd)^{2n})$ matrices, each of them being elementary or having the form $T\oplus {\bf I}_{s-r-1}$ for some $T\in SL_{r+1}(R).$
\end{enumerate}
\end{theorem}
The proof given in \cite{CCDHKS93} of this result is constructive. In the rest of this section, we will review some steps of it to make explicit the exponent  $\cO(n)$ which appears in (3).

Assume then that a unimodular matrix $F\in R^{r\times s}$ is given. In a preliminary step in \cite{CCDHKS93}, one has to make a linear change of coordinates, and then multiply $F$ by the unimodular matrix $A=(a_{ij})_{1\leq i,j\leq s}$ defined by
$$a_{ij}=\left\{\begin{array}{ccl}
x_n&\mbox{if}& i=j\leq r\\
1&\mbox{if}&j=i+1\\
1&\mbox{if}& i=s, j=1\\
0&&\mbox{everywhere else}
\end{array}
\right.,
$$
in such a way that the conditions of Assumption 2.8 in \cite{CCDHKS93} are satisfied, namely that the $r\times r$ minor of $F$ made by choosing the first $r$ columns is monic in all the variables $x_1,\ldots, x_n,$ and having total degree strictly larger than the degree of the remaining maximal minors of $F$. This may increase the value of $d$ in $1,$ so we will have to keep track of this in order to get an explicit bound.

Next we will have to deal with a version of Hilbert's Nullstellensatz presented in that paper. To do this, consider the $s\times s$ matrix $Y=(y_{ij})_{1\leq i,j \leq s}$ where each
of the $y_{ij}$ is a new indeterminate. Set $F_Y:= F\cdot Y\in (R\otimes\K[y_{ij}])^{r\times s}.$ Denote with $D_1$ (resp. $D_2$) the determinant of the $r\times r$ submatrix of $F_Y$ made by choosing its first $r$ columns (resp. the columns $1,\ldots, r-1, r+1$), and denote with $c\in\K[x_1,\ldots, x_{n-1}, y_{ij},\,1\leq i, j\leq s],$ the resultant as defined in \cite[Theorem 9.3]{wal78} of $D_1$ and $D_2$ with respect to $x_n$. From  \cite[Theorem 10.9]{wal78} we easily verify that
\begin{equation}\label{boundc}
\deg_{x_1,\ldots, x_{n-1}}(c)\leq \big(r(d+1)\big)^2,
\end{equation}
and moreover by applying the latter result and \cite[Theorem 9.6]{wal78} we deduce that there exist $A_1, A_2\in \K[x_1,\ldots, x_{n-1}, y_{ij},\,1\leq i,j\leq s]$ such that
\begin{equation}\label{bezc}
c=A_1D_1+A_2D_2, \ \mbox{with} \ \deg_{x_1,\ldots, x_{n-1}}(A_1, A_2)\leq \big(r(d+1)\big)^2.
\end{equation}
\begin{lemma}\cite[Lemma 4.4]{CCDHKS93}\label{keyn}
For all $\xi\in\K^{n-1},$ there exists ${\bf y}_\xi\in\K^{s\times s}$ such that $c(\xi, {\bf y}_\xi)\neq0$
\end{lemma}
With this result in hand, we can prove the following effective version of Hilbert's Nullstellensatz.
\begin{proposition}\label{nellefective}
There exist matrices ${\bf y}^1,\ldots, {\bf y}^n\in\K^{s\times s}$ such that
\begin{equation}\label{idd}
\langle c(x,{\bf y}^1),\ldots, c(x, {\bf y}^n)\rangle =R.
\end{equation}
\end{proposition}
\begin{proof}
Denote with $\overline{\K}$ the algebraic closure of $\K.$
Start by picking any $\xi_1\in\K^{n-1},$ and set ${\bf y}^1\mapsto y(\xi_1).$ Thanks to Lemma \ref{keyn} we have that $c(x, {\bf y}^1)\neq0,$ and hence the variety defined by its zeroes is a hypersurface in $\overline{\K}^{n-1}.$

Denote with $W_1,\ldots, W_\ell$ the irreducible components of maximal dimension of this hypersurface, and pick $\chi_1,\ldots, \chi_\ell\in\overline{\K}^{n-1}$ such that $\chi_j\in W_j$. By Lemma \ref{keyn}, none of the polynomials $c(\chi_j, y_{ij})\in\K[y_{ij}]$ can be identically zero, so the product $\prod_{j=1}^\ell c(\chi_j, y_{ij})$ also is non-zero. As $\K$ is infinite, we can choose
${\bf y}^2\in\K^{s\times s}$ such that $\prod_{j=1}^\ell c(\chi_j, {\bf y}^2)\neq0.$ With this choice, we have that the variety defined by the zeroes of $c(x,{\bf y}^1)$ and $c(x, {\bf y}^2)$ cannot have components of codimension $1$ in $\overline{\K}^{n-1}$ as the latter polynomial cannot vanish identically in any of the $W_j$.

The same argument can be applied recursively as follows: given $c(x,{\bf y}^1),\ldots, c(x,{\bf y}^i)$ such that the set of common zeroes of these polynomials has irreducible components of dimension at most $n-1-i,$ there exist ${\bf y}^{i+1}\in\K^{s\times s}$ such that $c(x,{\bf y}^{i+1})$ cuts properly every component of maximal dimension of this algebraic set. We will eventually arrive to the situation where  the system $c(x,{\bf y}^1),\ldots, c(x, {\bf y}^n)$  defines the empty variety in $\overline{\K}^{n-1}.$ Hilbert's Nullstellensatz then imply \eqref{idd}.
\end{proof}
We will also need the following refinement of the Effective Nullstellensatz for the degrees of polynomials involving a B\'ezout identity.
\begin{proposition}\label{25}
Let $c(x,{\bf y}^1),\ldots, c(x,{\bf y}^n)$ be as in \eqref{idd}. Then one can have
$$x_n=a_1 \, c(x,{\bf y}^1)+\ldots+a_n\, c(x, {\bf y}^n)
$$
with $a_1,\ldots, a_n\in R$, and $\deg\big(a_i\,c_i(x,{\bf y}^i)\big)\leq 2\big(r(d+1)\big)^{2(n-1)}.$
\end{proposition}
\begin{proof}
By the Effective Nullstellensatz (Theorem 1.1 in \cite{jel05}), there exist $b_1,\ldots, b_n\in\K[x_1,\ldots, x_{n-1}]$ such that
$1=b_1\,c(x,{\bf y}^1)+\ldots+b_n\,c_n(x, {\bf y}^n),$ with $\deg(b_i\,c(x,{\bf y}^i))\leq 2(r(d+1))^{2(n-1)}-1.$ The claim now follows by multiplying by $x_n$ both sides of this equality.
\end{proof}
\begin{theorem}\label{explicitqs}
Assume that $F\in R^{r\times s}$ is unimodular, with $r\leq s.$ Then, the matrix $U\in R^{s\times s}$ of Theorem \ref{eqs} can be computed with
$$\deg(U)\leq 3n^2(r(d+1))^{2n}.$$
\end{theorem}
\begin{proof}
We start by following the steps of the algorithms given in the proof of Proposition 4.1 (Procedure 4.3) in \cite{CCDHKS93} to compute a matrix $U_n$ which ``eliminates'' the variable $x_n$ from $F$ by evaluating it to zero, i.e. $F\cdot U_n=F|_{x_n=0}.$
\begin{itemize}
\item In their step $1$, their number $N$ can be replaced by $n$ thanks to Proposition \ref{nellefective}.
\item In their step $2$, we have $\deg(a_i\,c_i(x,{\bf y}^i))\leq 2\big(r(d+1)\big)^{2(n-1)}$ thanks to Proposition \ref{25}.
\item The degree of what is called $E_k$ in their step $3$ is  bounded by
$$\deg(E_k)\leq r(d+1)(1+r(d+1))2\big(r(d+1)\big)^{2(n-1)}\leq 3\big(r(d+1)\big)^{2n}.
$$
\item To compute the unimodular matrix $U_n$ (matrix $M$ in their notation), one has to multiply $N(=n)$ of these matrices $E_k,$ so we have that
$$\deg(U_n)\leq 3n\big(r(d+1)\big)^{2n}.$$
\end{itemize}
By applying this process recursively and eliminating all the variables, we see that the unimodular matrix $U$ of Theorem \ref{eqs} can be computed as a product of matrices $U_n\cdot U_{n-1}\ldots U_1\cdot U_0$, where for $i>0,$ each $U_i$ eliminates the variable $x_i,$ and $U_0\in\K^{s\times s}$ is a matrix of scalars. So, we have then that
$$\deg(U)\leq 3n^2\big(r(d+1)\big)^{2n},
$$
as claimed.
\end{proof}

\bigskip
\section{Syzygies and unimodularity}\label{3}
In this section, we will relate the matrices converting $a_1,\ldots, a_m$ into $p, q$, and vice versa, with unimodular matrices. This will allow us to use the Effective Quillen-Suslin Theorem \ref{eqs} to produce an $R$-basis of $\mbox{Syz}(a_1,\ldots, a_m)$ of controlled degree.

Let then $a_1,\ldots, a_m,\, p,\, q\in R$ be such that \eqref{1} holds. The syzygy module of the sequence $(a_1,\ldots, a_m)$ is defined as
$$\mbox{Syz}(a_1,\ldots, a_m):=\{(u_1,\ldots, u_m)\in R^m|\,u_1a_1+\ldots+u_ma_m=0\}\subset R^m.
$$
Note that we are not claiming that $\gcd(p,q)=1,$ but this can be assumed w.l.o.g. as $\mbox{Syz}(a_1,\ldots, a_m)$ does not change after removing a common factor of all these polynomials (which would be a common factor of $p$ and $q$ thanks to \eqref{1}) ).

From \eqref{1}, we deduce that there exist matrices $M\in R^{2\times m}$ and $N\in R^{m\times 2}$ such that
\begin{equation}\label{MyN}
(a_1\ a_2\ \ldots \ a_m)\cdot N= (p\ q), \ \ (p\ q) \cdot M = (a_1\ a_2 \ \ldots  \ a_m).
\end{equation}
In principle, there are infinite matrices $M$ and $N$ that satisfy \eqref{MyN}. The results that we prove in the sequel hold for \emph{any} of these choices. 
Denote with $K$ the $2\times2$ matrix which is the product $M\cdot N$.

\begin{lemma}\label{hunno}
Assuming that $\gcd(p,q)=1$, there exist $e,f\in R$ such that
\begin{equation}\label{kristina}
K=\begin{pmatrix} 1-e\cdot q & f\cdot q\\
e\cdot p&1-f\cdot p\end{pmatrix}.
\end{equation}
\end{lemma}
\begin{proof}
Write $K=\begin{pmatrix}\alpha&\beta\\ \gamma&\delta\end{pmatrix}.$
From \eqref{MyN} we deduce straightforwardly that

$$(p\ q)\cdot K=(p \ q)\cdot M\cdot N=(p\ q).$$
So, we have
$$\left\{
\begin{array}{ccl}
p&=&\alpha\cdot p+\gamma\cdot q\\
q&=&\beta\cdot p+\delta\cdot q
\end{array}
\right.
.$$
From here we deduce that
$$\left\{
\begin{array}{ccl}
(1-\alpha)p&=&\gamma\cdot q\\
(1-\delta)q&=&\beta\cdot p
\end{array}
\right.
,$$
so there exist $\lambda,\tilde\lambda$ in $Q(R)$, the field of fractions of $R$, such that
$$\left\{\begin{array}{ccl}
1-\alpha&=&\lambda\cdot q\\
\gamma&=&\lambda\cdot p\\
1-\delta&=&\tilde\lambda\cdot p\\
\beta&=&\tilde\lambda\cdot q
\end{array}
\right.
.$$
As $\lambda\cdot q$ and $\lambda\cdot p$ are elements of $R$ and $\gcd(p,q)=1,$ then we deduce that $\lambda\in R$. The same happens with $\tilde\lambda$. The claim follows by setting $e\mapsto \lambda,\,f\mapsto\tilde\lambda.$
\end{proof}

We cannot claim that $K$ is a unimodular matrix. As a matter of fact, with the notation above we have that $\det(K)=1-e\cdot q-f\cdot p,$ which is an element of $\K$ if and only if $e\cdot q+f\cdot p\in \K$. If this is the case and $\det(K)\neq1,$ then $I$ is principal, and the Quillen-Suslin Theorem \ref{eqs} shows that $\mbox{Syz}(a_1,\ldots, a_m)$ is a free $R$-module, and gives bounds for the degrees of a basis of this module, which are better than those appearing in Theorem \ref{mt}.

In any case, we can modify the matrix $M$ so that we get $e=f=0.$ We start by denoting with
\begin{equation}\label{tildeK}
\tilde K:=\begin{pmatrix}
1-e\cdot q & f\cdot q &-q \\
e\cdot p&1-f\cdot p&p
\end{pmatrix}\in R^{2\times 3},
\end{equation}
the matrix which consists in adding to $K$ the column $\begin{pmatrix}-q\\ p\end{pmatrix}.$

\begin{lemma}\label{propi}
$\tilde K$ is a unimodular matrix. \end{lemma}
\begin{proof}
Indeed the $3$ maximal minors of $\tilde K$ are $p, q$ and $1-e\cdot q-f\cdot p.$ From here the claim follows straightforwardly.
\end{proof}

To connect $\tilde K$ with $M$ and $N$, we write them down explicitly:
\begin{equation}\label{NM}
N=\begin{pmatrix}
b_1 & c_1 \\
b_2 & c_2 \\
\vdots & \vdots \\
b_m & c_m
\end{pmatrix} \ \mbox{and} \  M=\begin{pmatrix} d_1 & d_2  & \ldots & d_m \\ e_1 & e_2 & \ldots  & e_m\end{pmatrix}.
\end{equation}

\begin{proposition}\label{Nu}
$N$ is a unimodular matrix, and so is $\tilde M,$ where
\begin{equation}\label{tildem}
\tilde M:= \begin{pmatrix} d_1 & d_2  & \ldots & d_m &-q \\ e_1 & e_2 & \ldots  & e_m&p \end{pmatrix} \in R^{2\times (m+1)}.
\end{equation}
\end{proposition}

\begin{proof}
Set
\begin{equation}\label{tilden}
\tilde N=\begin{pmatrix}
b_1 & c_1 &0 \\
b_2 & c_2 &0 \\
\vdots & \vdots \\
b_m & c_m &0 \\
0& 0& 1
\end{pmatrix}\in R^{(m+1)\times 3}.
\end{equation}
We clearly have $\tilde M\cdot\tilde N=\tilde K.$ As $\tilde K$ is unimodular, so are $\tilde N$ and $\tilde M$ (this can be seen for instance by using the Cauchy-Binet formula  (\cite[\S 4.6]{BW89})for computing minors of a product of matrices). The fact that $N$ is unimodular follows just by noting that all the nonzero maximal minors of it are ---up to the sign--- the nonzero maximal minors of $\tilde N$.
\end{proof}

Note that Proposition \ref{Nu} states that {\emph any} matrix $N$ as in \eqref{MyN} is unimodular. That does not apply to $M,$ see for instance  Example \ref{aex}. However, one can always replace $M$ with a unimodular one as the following result shows.
\begin{proposition}\label{Mun}
$M$ can be chosen as in \eqref{MyN} to be unimodular.
\end{proposition}

\begin{proof}
Let $e,f\in R$ be such that  \eqref{kristina} holds. If $(e,f)=(0,0)$ we are done. If not, because $N$ is unimodular, its rows generate $R^2$ as an $R$-module, so there exist $x_1\ldots x_m\in R$ such that
\begin{equation}\label{oss}
x_1\cdot (b_1 \ c_1)+ x_2\cdot (b_2\ c_2)+\ldots +x_m\cdot (b_m\ c_m)=(e\ -f).
\end{equation}
Set then
\begin{equation}\label{Mp}
M'=M+\begin{pmatrix} x_1q& x_2q&\ldots &x_mq \\ -x_1p&-x_2p&\ldots &-x_mp\end{pmatrix}.
\end{equation}
We clearly have that $M'$ satisfies \eqref{MyN}, and an easy computation shows that
\begin{equation}\label{ortog}
M'\cdot N=\left(M+\begin{pmatrix} x_1q& x_2q&\ldots &x_mq \\ -x_1p&-x_2p&\ldots &-x_mp\end{pmatrix}\right)\cdot N=\begin{pmatrix}1&0\\ 0&1\end{pmatrix},
\end{equation}
thanks to \eqref{oss}. So, $M'$ is unimodular, as claimed.
\end{proof}
We conclude by showing a characterization of those ideals of grade $2$ having the property \eqref{1} via the matrix $M$.
\begin{theorem}\label{char}
For $a_1\ldots, a_m,\, p,\, q\in R$, we have that $\langle a_1,\ldots, a_m\rangle=\langle p, q\rangle$ if and only if there exists a unimodular matrix $M\in R^{2\times m}$ such that
$$(p \ q)\cdot M=(a_1\ldots a_m).
$$
\end{theorem}
\begin{proof}
The ``if'' part follows from Proposition \ref{Mun}. For the converse, apply Theorem \ref{eqs} to the unimodular matrix $M$, and let $U\in R^{m\times m}$ be such that
$M\cdot U=\begin{pmatrix} 1&0 &0& \ldots & 0 \\ 0&1& 0& \ldots &0\end{pmatrix}.$ Set $N\in R^{m\times 2}$ to be the matrix consisting of the first two columns of $U$. We then have that \eqref{MyN} holds and so \eqref{1}, which proves the claim.
\end{proof}
 \begin{remark}
 The role of $M$ and $N$ are different in the characterization of ideals satisfying \eqref{1}. Indeed, having $N$ unimodular is not enough to characterize these ideals, as for instance we may have $a_1,\ldots, a_m\in R$ with $m\geq3$ be such that $a_3\notin\langle a_1, a_2\rangle.$ If we pick set $N=\begin{pmatrix} 1&0\\ 0&1\\ 0&0\\ \vdots&\vdots\\ 0&0\end{pmatrix},$ then it is clear that we will never find an $M$ such that \eqref{MyN} holds.
 \end{remark}
 \bigskip
 \section{Algorithms}\label{alg}
We will now exhibit algorithms to compute $R$-bases of $\mbox{Syz}(a_1,\ldots, a_m)$ by applying the Effective Quillen-Suslin Theorem \ref{eqs} to $\tilde{M},\,M'$ or $N$.

\smallskip
\subsection{Working with $\tilde{M}$}
With notation as above, from Theorem \ref{eqs} we deduce that there exists a square unimodular matrix $U\in R^{(m+1)\times (m+1)}$ such that
\begin{equation}\label{QS}
\tilde M \cdot U =\begin{pmatrix}
1&0&0&\ldots &0\\
0&1&0 &\ldots &0
\end{pmatrix}.
\end{equation}
Assume w.l.o.g. that $\det(U)=1.$ From \eqref{MyN} we deduce that
$(p \ q)\cdot \tilde M= (a_1\ldots a_m \ 0),$ and hence we must have
\begin{equation}\label{falabella} (p \ q ) \cdot \tilde M\cdot U=
(p \ q \ 0 \ldots 0)
= (a_1\ldots a_m\ 0)\cdot U
\end{equation}
Let $\widehat{U}\in R^{m\times (m-1)}$ the submatrix of $U$ consisting in removing its first two columns and its last row. From \eqref{falabella} we deduce that the columns of $\widehat{U}$ are syzygies of $(a_1,\ldots, a_m).$ Our main result is the following.
\begin{theorem}\label{mt}
The columns of $\widehat{U}$ are an $R$-basis of $\mbox{Syz}(a_1,\ldots, a_m).$
\end{theorem}
\begin{proof}
Let $U^1,\ldots, U^{m+1}$ be the columns of $U$, and denote with $\tilde{U}\in R^{(m+1)\times m}$ the matrix whose columns are $-q\cdot U^1+p\cdot U^2 ,\ U^3,\ldots, U^{m+1},$ in this order.
By applying Cramer's rule to the last equality of \eqref{falabella}, we deduce that the signed maximal minors of this matrix are $a_1,\ldots, a_m, 0$. Hence, the ideal generated by these maximal minors is $I,$ which has grade $2$ by our initial assumptions. By applying then the converse of the Hilbert-Burch Theorem \ref{hb}, we deduce then that the columns of $\tilde{U}$ are a basis of $\mbox{Syz}(a_1,\ldots, a_m, 0)$.

We claim now that the first two rows of $U^{-1},$ the inverse matrix of $U$, is equal to $\tilde{M}.$  Indeed, denote then by $U^{-1}_{2\times(m+1)}$ the submatrix of $U^{-1}$ made by these rows. From \eqref{QS} we deduce that
$$\tilde M\cdot U=U^{-1}_{2\times (m+1)}\cdot U,
$$
and as $U$ is invertible, the claim follows. In particular, the last column of $U^{-1}$ is of the form
$\begin{pmatrix} -q \\ p \\ r_3 \\ \vdots \\ r_{m+1}\end{pmatrix}$ for suitable $r_3,\ldots, r_{m+1}\in R.$
From the identity $U\cdot U^{-1}={\bf I}_{m+1},$  we deduce that $$-q\cdot U^1+p\cdot U^2+r_3\cdot U^3+\ldots +r_{m+1}\cdot U^{m+1}=\begin{pmatrix}
0\\ \vdots\\ 0\\ 1
\end{pmatrix}.$$ This implies that if we perform the following column operation in $\tilde{U}$: to its first column (which is equal to $-q\cdot U^1+p\cdot U^2$) we add $r_3U^3+\ldots +r_{m+1}U^{m+1},$ the fact that its columns are an $R$-basis of $\mbox{Syz}(a_1,\ldots, a_m,0)$ remains unchanged, but now the first column of the modified matrix is equal to
$\begin{pmatrix}
0\\ \vdots\\ 0\\ 1
\end{pmatrix}.$

From here, we can perform new column operations in $U^3, U^4,\ldots, U^{m+1}$ in such a way that the  last row of $\tilde{U}$ equals to $(1 \ 0 \ldots \ 0),$ and the other coefficients of this matrix have not changed.
So, we have shown that  $\mbox{Syz}(a_1,\ldots, a_m,0)$ has an  $R$-basis of the form
$$\left(\begin{array}{cc}
0& \widehat{U}\\
1&0
\end{array}\right).
$$
From here the claim follows straightforwardly.
\end{proof}

\smallskip
\subsection{Working with a unimodular $M$}
If $M$ is already unimodular (which is for instance the case of the matrix $M'$ defined in \eqref{Mp}, although  we are not requiring that $M\cdot N={\bf I}_2$ as in \eqref{ortog}), we can apply directly the Effective Quillen-Suslin Theorem \ref{eqs} to $M$ and obtain a unimodular $U^*\in R^{m\times m}$ such that
\begin{equation}\label{MQS}
M \cdot U^* =\begin{pmatrix}
1&0&0&\ldots &0\\
0&1&0 &\ldots &0
\end{pmatrix}.
\end{equation}
From \eqref{MyN} we now get that
$(p \ q)\cdot M= (a_1\ldots a_m),$ and hence
\begin{equation}\label{sisy}
(p \ q ) \cdot M\cdot U^*=
(p \ q \ 0 \ldots 0)
= (a_1\ldots a_m)\cdot U^*.
\end{equation}
Denote the columns of $U^*$ with $U^{*1}\  U^{*2}\ldots,\,U^{*m}$. Let $\widehat{U^*}\in R^{m\times (m-1)}$ be the matrix
\begin{equation}\label{whu}
\widehat{U^*}=\left( q U^{*1}-p U^{*2}, \ U^{*3},\ldots, \,U^{*m}\right).
\end{equation}

By applying Cramer's rule to \eqref{sisy}, and using the converse of the Hilbert-Burch Theorem \ref{hb}, we deduce straightforwardly that
\begin{theorem}\label{smt}
The columns of $\widehat{U^*}$ defined in \eqref{whu} are an $R$-basis of $\mbox{Syz}(a_1,\ldots, a_m).$
\end{theorem}

\smallskip
\subsection{Working with $N$}\label{22}
We can also work directly with the unimodular matrix $N$ from \eqref{MyN} and construct an $R$ basis of $\mbox{Syz}(a_1,\ldots, a_m)$ as follows:
denote with $N^* \in R^{m\times m}$ a matrix of determinant equals to $1$ such that  $N$ has its first $2$ columns. This can be done by applying  the Quillen-Suslin Theorem  \ref{eqs} to $N^{\bf t}.$ The matrix $N^*$ can be taken to be the inverse of matrix $U$ in this claim. Denote its columns with $N^{*1},\ldots, N^{*m}.$

From \eqref{MyN}, we clearly have that $(a_1 \ldots a_m) \cdot N^{*1}=p,$ and $(a_1\ldots a_m) \cdot N^{*2}=q.$  As for all $i=1,\ldots, m,$ we have that $(a_1\ldots a_m) \cdot N^{*i}\in\langle a_1,\ldots, a_m\rangle=\langle p, q\rangle,$ for $i=3,\ldots, m$ we write $(a_1\ldots a_m) \cdot N^{*i}=\lambda_i p+\delta_i q$ for suitable $\lambda_i,\,\delta_i\in R.$

Perform then the following elementary column operations in $N^*$: for $i=3,4,\ldots, m,$ replace column $N^{*i}$ with $N^{*i}-\lambda_i N^{*1}-\delta_i N^{*2}.$ Call the remaining matrix $N^{**}.$ By construction:
\begin{itemize}
\item $\det(N^{**})=1,$ i.e. $N^{**}$ is unimodular;
\item $(a_1\ldots a_m)\cdot N^{**}=(p \ q \ 0 \ldots 0)$
\end{itemize}
Denote then with $N^{**1}, N^{**2},\ldots$ the columns of $N^{**}$, and let $\widehat{N}\in R^{m\times(m-1)}$ the matrix whose columns are
\begin{equation}\label{N}
\widehat{N}= \big(q N^{**1}-p N^{**2} \ N^{**3} \ldots N^{**m} \big)
\end{equation}
\begin{theorem}\label{NN}
The columns of $\widehat{N}$ are an $R$-basis of $\mbox{Syz}(a_1,\ldots, a_m).$
\end{theorem}

\begin{proof}
As before, by taking into account that $(a_1\ldots a_m)\cdot N^{**}=(p \ q \ 0 \ldots 0),$ and applying Cramer's rule to this matrix, it is easy to see that the signed maximal minors of $\widehat{N}$ are $a_1,\ldots a_m$. The converse of the Hilbert-Burch Theorem \ref{hb} then proves the claim.
\end{proof}

\smallskip
\subsection{Relations among the bases}\label{rell}
In the following, we will show  that if one picks convenient unimodular matrices in the process of computing the several $R$-basis of $\mbox{Syz}(a_1,\ldots, a_m)$ described above, the ansatz is essentially the same.
We begin by proving the following straightforward relation between $U^*$ and $N^*$ if $M$ is already unimodular.
\begin{proposition}\label{antrior}
Suppose that  $M\cdot N={\bf I}_2$ (in particular, we have that $M$ is unimodular). Then, one can find  a unimodular $m\times m$ matrix which can be used as $U^*$  in \eqref{MQS} and also as $N^{**}$ from \S \ref{22}.
 Moreover, we get $\widehat{U^*}=\widehat{N},$ i.e. the bases from Theorems \ref{smt} and \ref{NN} coincide.
\end{proposition}

\begin{proof}
Start by picking any $U_0^*\in R^{m\times m}$ satisfying \eqref{MQS}. Denote its columns with $U_{0}^{*1}\ldots U_{0}^{*m}.$  As we have
$$M\cdot \big( U_{0}^{*1}\ U_{0}^{*2}\big)=M\cdot N=\begin{pmatrix} 1&0\\ 0&1\end{pmatrix},$$
we deduce that  each of the columns of $\big( U_{0}^{*1}\ U_{0}^{*2}\big)-N$ belongs to the right-kernel of $M.$  Denote also with $N^1$ and $N^2$ the columns of $N$. As $M$ is unimodular, its  right-kernel is a free $R$-module, with basis $U_{0}^{*3},\ldots, U_{0}^{*m}.$ So there exist $y_{i3},\ldots, y_{im}\in R$ such that
$$U_{0}^{*i}-N^i=y_{i3} U_{0}^{*3}+\ldots + y_{im} U_{0}^{*m}, \ i=1,  2$$
Set now $U^*$ to be the matrix obtained from $U_0^*$ by substracting to the column $i$ the following linear combination of columns:
$y_{i3}U_{0}^{*3}+\ldots+y_{im}U_{0}^{*m},\, i=1, 2.$  Clearly  $U^*$ is unimodular, and has $N$ as its first two columns, so the first part of the claim follows.

To see the second part, note that  as we have
$$(a_1\ldots a_n)\cdot N^*=(p \ q)\cdot M\cdot U^*=(p \ q \ 0 \ldots 0),
$$
we deduce that $\lambda_i,\,\delta_i,\,i=3,\ldots, m$ from \S \ref{22} are all equal to zero, and hence the matrix $N^{**}$ defined in that section equals to $N^*$ (which is equal to $U^*$). As the process to convert $N^{**}$ into $\widehat{N}$ and $U^*$ into $\widehat{U^*}$ is the same (replace the first two columns with $q$ times the first column minus $p$ times the second column), we deduce straigthforwardly that $\widehat{U^*}=\widehat{N},$ which concludes with the proof of the proposition.
\end{proof}

If $M$ is not unimodular anymore, we will have to work with matrices $\tilde{M}$ and $\tilde{N}$ which were defined in  \eqref{tildem} and \eqref{tilden} respectively.  The following result then holds.
\begin{proposition}
The matrices $U$ and $N^*$ from  \eqref{QS} and \S \ref{22} respectively can be chosen in such a way that the latter is a submatrix of the first, and that
 $\widehat{U^*}=\widehat{N},$ i.e. the bases from Theorems \ref{smt} and \ref{NN} coincide.
\end{proposition}
Note that we are not having any requirement about the result of the product between $\tilde{M}$ and $\tilde{N}$ as it may be hinted by the situation in Proposition \ref{antrior}.

\begin{proof}
Denote the columns of $N$ with $N^1,\,N^2,$ and set

\begin{equation}\label{tilden*}
{\tilde N}^*=\begin{pmatrix}
N^1& N^2& qN^1-pN^2\\
-e& f &1-eq-fp
\end{pmatrix}\in R^{(m+1)\times 3},
\end{equation}
with $e,f\in R$ as in \eqref{hunno}.
As $\tilde{M}\cdot\tilde{N}=\tilde{K},$ with $\tilde{K}$ defined in \eqref{tildeK}, we deduce that
$$\tilde{M}\cdot\tilde{N}^*=\begin{pmatrix} 1&0&0\\ 0&1&0\end{pmatrix}.$$
This implies that $\tilde{N}^*$ is unimodular, and hence can be extended to an invertible $U\in R^{(m+1)\times (m+1)}$ such that \eqref{QS} holds and having $\tilde{N}^*$ as its first columns. We have then that
$$U=\begin{pmatrix}
N^1& N^2& qN^1-pN^2 & N^4&\ldots &N_{m+1}\\
-e& f &1-eq-fp & r_4&\ldots &r_{m+1}
\end{pmatrix}
$$
for suitable columns $N^4,\ldots, N^{m+1}$ in $R^{m\times 1}$ and suitable $r_4,\ldots, r_{m+1}\in R.$ As by performing elementary operations in the first columns of $U$ we can get $\begin{pmatrix} 0\\  \vdots\\ 0\\ 1\end{pmatrix},$ we deduce then that, by picking $N^*=(N^1 \ N^2\ N^4 \ldots N^{m+1})\in R^{m\times m},$
\begin{itemize}
\item $N^*$ is unimodular and  has as first two columns $N^1$ and $N^2,$ i.e. it is a valid $N^*$ in the sense of the algorithm described in \S \ref{22}.
\item Following the steps of that algorithm, $N^{**}=N^*.$
\item $\widehat{N}$ from \eqref{N} equals to  $ \big(qN^1-pN^2 \ N^4\ldots N^{m+1}\big)=\widehat{U}$ from Theorem \ref{mt}.
\end{itemize}
This concludes with the proof of the Proposition.
\end{proof}

\bigskip
\section{Degree bounds and proof of the main theorem}\label{4}
In this section, we will apply our explicit bound given in Theorem \ref{explicitqs} to give bounds for $\deg(\widehat{U})$ from Theorem \ref{mt}, $\deg(\widehat{U^*})$ from Theorem \ref{smt}, and $\deg(\widehat{N})$ from Theorem \ref{NN}. These bounds are given in terms of the degrees of both the input polynomials, but also of the transition matrices. We will show also some relations among the input bounds, and conclude by proving Theorem \ref{mtt}  from the Introduction.

Recall that we have $\delta_0$ being a bound for the degrees of $p$ and $q,$ and $\delta_a$ being a bound on the degrees of $a_1,\ldots, a_m.$ In addition, we set
 $\delta_N$ (resp. $\delta_M$) being  a bound for the degrees of the elements in $N$ (resp. $M$).

 \begin{proposition}\label{41}
 With notation as above, if $\delta\geq\max\{\delta_M,\,\delta_0\},$ we have that
 $$\deg(\widehat{U})\leq 3n^24^n(\delta+1)^{2n}.
 $$
 \end{proposition}
 \begin{proof}
 Note that $\widehat{U}$ is a submatrix of the unimodular matrix $U$. The result is then consequence of Theorem \ref{explicitqs} with $r=2.$
 \end{proof}

 \begin{proposition}\label{41b}
 With notation as above, if  $M$ is unimodular, we have that
 $$\deg(\widehat{U^*})\leq 3n^24^n(\delta_M+1)^{2n}+\delta_0.
 $$
 \end{proposition}
 \begin{proof}
 By applying Theorem \ref{explicitqs} to the unimodular matrix $M$ we obtain $U^*$ as in \eqref{MQS} with
 $$\deg(U^*)\leq 3n^24^n(\delta_M+1)^{2n}.$$
To get $\widehat{U^*}$ we only have to modify the first column of $U^*$ and multiply a couple of columns of the latter by $-q$ and $p$. From here, the claim follows straigthforwardly.
 \end{proof}

 Interestingly, the computation of a basis of syzygies starting from $N$ will give us degree bounds which can differentiate the degrees of the different actors involved. The drawback is that this bound depends also on $m$, the number of elements of the sequence. Sometimes this leads to lower bounds as in the case of parametric curves, see Section \ref{final}.

 \begin{proposition}\label{42}
 With notation as above, we have that
 $$\deg(\widehat{N})\leq\delta_0+\delta_N+\delta_M+3mn^24^n(\delta_N+1)^{2n}.
 $$
 \end{proposition}
 \begin{proof}
 By applying the Effective Quillen-Suslin procedure to the unimodular matrix $N,$ we deduce that the matrix $U$ of Theorem \ref{eqs} has degree bounded by $3n^24^n(\delta_N+1)^{2n},$ thanks to Theorem \ref{explicitqs} with $r=2.$

 The matrix $N^*$ from \S \ref{22} is then the inverse of $U,$ and can be computed by using cofactors. Hence, its entries have degree bounded by
 $$3mn^24^n(\delta_N+1)^{2n}.
 $$
 As we have $(p\ q)\cdot M=(a_1\ a_2\ldots a_m)$ from \eqref{MyN}, we get
 $$(a_1\ldots a_m)\cdot {N^*}^i=\lambda_ip+\delta_iq=(p\ q)\cdot M\cdot {N^*}^i,$$
 and deduce then that $\deg(\lambda_i),\,\deg(\delta_i)\leq\delta_N+\delta_M$ for all $i=3,\ldots, m.$ By construction, we have that
 $$\deg(N^{**})\leq\delta_N+\delta_M+3mn^24^n(\delta_N+1)^{2n}.
 $$
 To pass from $N^{**}$ to $\widehat{N},$ the matrix encoding a basis of $\mbox{Syz}(a_1,\ldots, a_m),$ we only need to modify the first column of $N^{**},$ and we have
 $$\deg(\widehat{N})\leq\delta_0+\delta_N+\delta_M+3mn^24^n(\delta_N+1)^{2n},
 $$
 as claimed.
 \end{proof}

 \smallskip
 \subsection{From $\tilde{M}$ to an unimodular $M$}
The bounds on the degrees of the bases computed via $\tilde{M}$ or $M$ if the latter is unimodular are not very comparable, for instance they depend on whether $\delta_0>\delta_M$ or viceversa. We will compute for completion of our study of the degrees appearing in these matrices a bound on the unimodular matrix $M'$ defined in \eqref{Mp} in terms of the input degrees of the problem.
\begin{proposition}
With notation as above, a matrix $M'$ as in \eqref{Mp} can be computed with
$$\deg(M')\leq 3n^24^n(\delta_N+1)^{2n}+\delta_M+\delta_N+\delta_0.
$$
\end{proposition}
 \begin{proof}
 Note that $e$ and $f$ from \eqref{kristina} have degrees bounded by $\delta_M+\delta_N$. Let $U_N\in R^{m\times m}$ the unimodular matrix such that
 $N^{\bf t}\cdot U_N=({\bf I}_2\, {\bf0})\in R^{2\times m}.$ By Theorem \ref{explicitqs}, $U_N$ can be computed with $\deg(U_N)\leq 3n^24^n(\delta_N+1)^{2n}.$
As usual, denote the columns of $U_N$ with $U_N^1,\ldots U_N^m.$ If we set,
$$\begin{pmatrix} x_1\\ x_2\\ \vdots\\ x_m\end{pmatrix}:=eU_N^1+fU_N^2$$
we deduce that \eqref{oss} is satisfied. By computing explicitly, we get $$\deg(x_i)\leq 3n^24^n(\delta_N+1)^{2n}+\delta_M+\delta_N.$$
To compute $M'$ from $M$ as in \eqref{Mp}, we have to add to the latter a matrix with coefficients then bounded by $3n^24^n(\delta_N+1)^{2n}+\delta_M+\delta_N+\delta_0$. From here, the claim follows straightforwardly, as this bound is larger than $\delta_M,$ which bounds the degree of $M$.
 \end{proof}

\smallskip
\subsection{Relations among bounds}
So far, we have
\begin{itemize}
\item $\delta_M$, a bound on the degree of the elements of $M$;
\item $\delta_N$, a bound on the degree of the elements of $N$;
\item $\delta_0$, a bound on the degrees of $p$ and $q;$
\item $\delta_a,$ a bound on the degrees of $a_1,\ldots, a_m.$
\end{itemize}

Is there any relation among these bounds? Clearly, from \eqref{MyN} we deduce straightforwardly that
\begin{itemize}
\item given $\delta_a$ and $\delta_N$, one can set $\delta_0:=\delta_a+\delta_N$;
\item given $\delta_0$ and $\delta_M$, one can set $\delta_a:=\delta_0+\delta_M.$
\end{itemize}

The following relation is less subtle.
\begin{proposition}\label{43}
Given $\delta_0$ and $\delta_a$, there exists a matrix $M$ such that one can take $\delta_M:=\delta_0^2+\delta_a$
\end{proposition}
\begin{proof}
As we are assuming $\gcd(p,q)=1,$ we get these polynomials are an affine complete intersection in $\K^n$. The result then follows straightforwardly from Corollary 5.2 in \cite{DFGS91}.
\end{proof}

The connection between $M$ and $N$ given via Theorem \ref{char} gives a bound for $\delta_N$ in terms of $\delta_M$ and $\delta_0$ but not a very optimal one.
\begin{proposition}
Given $\delta_M$ and $\delta_0,$ there exists a matrix $N$ such that one can take $\delta_N=3n^24^n(\max\{\delta_0,\,\delta_M\}+1)^{2n}.$
\end{proposition}
\begin{proof}
A possible matrix $N$ can be taken by using the first columns (except the last row) of a unimodular $U\in R^{(m+1)\times (m+1)}$ such that \eqref{QS} holds. Thanks to Theorem \ref{explicitqs}, we have that the degree of $N$ is bounded by $3n^24^n(\max\{\delta_0,\,\delta_M\}+1)^{2n},$ which proves the claim.
\end{proof}
In the zero-dimensional case, one can have a sharper bound for $\delta_N.$
\begin{proposition}\label{44}
If the ideal $I$ is zero-dimensional, given $\delta_0$ and $\delta_a,$ there exists a matrix $N$ such that one can take $\delta_N=2\delta_a^2+\delta_a+\delta_0$.
\end{proposition}
\begin{proof}
This follows essentially from Theorem 2.5 in \cite{has09}.
\end{proof}
We conclude this section by giving the proof of the main result announced in the Introduction
\smallskip
\subsection{Proof of Theorem \ref{mtt}}\label{4.2}
The first algorithm essentially consists in computing the matrix $U$ from \eqref{QS}, and extract the submatrix $\widehat{U}$ which ---thanks to Theorem \ref{mt}--- encodes an $R$-basis of $\mbox{Syz}(a_1,\ldots, a_m).$

A degree bound for $\widehat{U}$ is given in Proposition \ref{41} in terms of the degrees of $\delta_M$ and $\delta_0$. By using Proposition \ref{43}, we can replace this bound by $\delta_0^2+\delta_a,$ and then the first part of the claim follows straightforwardly.

For the second part, we will work with the matrix $N$ as in \S \ref{22}. The algorithm proposed there computes $\widehat{N}$ which ---thanks to Theorem \ref{NN}--- encodes also an $R$-basis of $\mbox{Syz}(a_1,\ldots,a_m),$ having the degree bounds given in Proposition \ref{42}. Note that we are not assuming yet that $I$ is zero-dimensional. This would be used to replace $\delta_N$ and $\delta_M$ with the bounds given in Proposition \ref{43} and \ref{44} to get \eqref{2}. $\qed$

\bigskip
\section{Examples}\label{5}
We present here the computation of two examples. The first one is the running example in \cite{CDM20} and has already $M$ being unimodular, while the second does not. To be consistent with the notation of \cite{CDM20}, we label the variables as $s, t$. So in both examples we have that $n=2$.

\subsection{Example 4.1 in \cite{CDM20}}
Here we have $m=4,\, p= t-s+2,\, q=s^2+1$ and
$$
\left\{\begin{array}{ccl}
a_1(s,t)&=& 11-4s+3s^2+4t\\
a_2(s,t)&=&5-4s+2s^2+4t-2st+t^2\\
a_3(s,t)&=&1+3s^2-s^3+s^2t\\
a_4(s,t)&=&7-3s+s^2+3t.
\end{array}
\right.
$$
As it was shown in \cite{CDM20}, one can take for this case
$$M=\left(\begin{array}{cccc}
4& t-s+2& s^2& 3\\
3& 1& 1& 1
\end{array}
\right).
$$ In addition, a simple $N$ is the following:
\begin{equation}\label{Nex1}
N=\left(
\begin{array}{cc}
-\frac15&\frac35\\
0&0\\
0&0\\
\frac35&-\frac45
\end{array}
\right).
\end{equation}
In this case, as the columns $1$ and $4$ of $M$ are already an invertible matrix in $\K[s,t].$ The same happens with the rows $1$ and $4$ of $N$.
By pivoting the $2\times2$ invertible submatrix of $M$,  it is easy to compute a matrix $4\times4$ unimodular matrix $U^*$   such that
$$M \cdot U^* =\begin{pmatrix}
1&0&0&0 \\
0&1&0 &0
\end{pmatrix},$$
we get:
$$
U^*=\left(
\begin{array}{cccc}
 -\frac{1}{5} & -\frac{s}{5}+\frac{t}{5}+\frac{2}{5} & \frac{s^2}{5}+\frac{s}{5}-\frac{t}{5}-\frac{2}{5} & \frac{s}{5}-\frac{t}{5}+\frac{1}{5}  \\
 0 & 1 & -1 & -1  \\
 0 & 0 & 1 & 0 \\
 \frac{3}{5} & \frac{3 s}{5}-\frac{3 t}{5}-\frac{6}{5} & -\frac{3 s^2}{5}-\frac{3 s}{5}+\frac{3 t}{5}+\frac{6}{5} & -\frac{3 s}{5}+\frac{3 t}{5}+\frac{2}{5}
\end{array}
\right),$$
To pass from $U^*$ to $\widehat{U^*}$ we proceed as in \eqref{whu} and get that
$$\widehat{U^*}=
\left(
\begin{array}{ccc}
-\frac{2 s^2}{5}+\frac{s t}{5}+\frac{9 s}{5}-\frac{7 t}{5}-3&  \frac{s^2}{5}+\frac{s}{5}-\frac{t}{5}-\frac{2}{5} & \frac{s}{5}-\frac{t}{5}+\frac{1}{5}  \\
s-t-2& -1 & -1  \\
0 &1 & 0  \\
\frac{6 s^2}{5}-\frac{6 s t}{5}-\frac{12 s}{5}+\frac{3 t^2}{5}+\frac{12 t}{5}+3 &-\frac{3 s^2}{5}-\frac{3 s}{5}+\frac{3 t}{5}+\frac{6}{5} & -\frac{3 s}{5}+\frac{3 t}{5}+\frac{2}{5}
\end{array}
\right).
$$
Thanks to Theorem \ref{smt}, the columns of $\widehat{U^*}$ encode a basis of $\mbox{Syz}(a_1, a_2, a_3, a_4)$. Note that what we obtained with this procedure is quite different than the basis obtained in \cite{CDM20}. For instance, the degree of this basis is $2,$ which is lower than the one obtained in that paper (equal to $5$).

Now we work with $N$. From \eqref{Nex1} it is easy to extend $N$ to a $4\times4$ unimodular matrix, we chose
$$N^*=\left(
\begin{array}{cccc}
 -\frac{1}{5} & \frac{3}{5} & 0 & 0 \\
 0 & 0 & -5 & 0 \\
 0 & 0 & 0 & 1 \\
 \frac{3}{5} & -\frac{4}{5} & 0 & 0 \\
\end{array}
\right).$$
To produce the matrix $N^{**}$ of the algorithm, we have to modify the columns $3$ and $4$ of $N^*$ by using $M$. We obtain
$$N^{**}=
\left(
\begin{array}{cccc}
 -\frac{1}{5} & \frac{3}{5} & s-t+1 & \frac{s^2}{5}-\frac{3}{5} \\
 0 & 0 & -5 & 0 \\
 0 & 0 & 0 & 1 \\
 \frac{3}{5} & -\frac{4}{5} & -3 s+3 t+2 & \frac{4}{5}-\frac{3 s^2}{5} \\
\end{array}
\right).
$$
Finally, to obtain $\widehat{N}$ we must replace the columns $1$ and $2$ of $N^{**}$ with the first column multiplied by $-q$ plus the second column multiply by $p$:
$$\widehat{N}=
\left(
\begin{array}{ccc}
 \frac{s^2}{5}-\frac{3 s}{5}+\frac{3 t}{5}+\frac{7}{5} & s-t+1 & \frac{s^2}{5}-\frac{3}{5} \\
 0 & -5 & 0 \\
 0 & 0 & 1 \\
 -\frac{3 s^2}{5}+\frac{4 s}{5}-\frac{4 t}{5}-\frac{11}{5} & -3 s+3 t+2 & \frac{4}{5}-\frac{3 s^2}{5} \\
\end{array}
\right).
$$
By Theorem \ref{NN}, the columns of $\widehat{N}$ encode another $R$-basis of $\mbox{Syz}(a_1,a_2,a_3,a_4).$ Note that the bases obtained via $\widehat{N}$ and $\widehat{U^*}$ are essentially different. For instance, the first column of $\widehat{N}$ is a relation that only involves $a_1$ and $a_4$, but nothing of this nature can be found from the columns of $\widehat{U^*}.$ This is because the relation $M\cdot N=\begin{pmatrix}1&0\\ 0&1\end{pmatrix}$ does not hold.

\bigskip
\subsection{Another Example}\label{aex}
Set now $m=4$ again, and $p=t+2s+1,\, q=-2t-s$ which are not under the conditions of the Shape Lemma as all the results in \cite{CDM20} are.  Consider the following sequence of polynomials:
$$\left\{
\begin{array}{ccc}
a_1&=& s+3s^2+t+4st-t^2\\
a_2&=&-s^2+t+t^2\\
a_3&=&s+2s^2-2t^2\\
a_4&=&1+s-t.
\end{array}
\right.
$$
In this case, we can take
$$M=\left(\begin{array}{cccc}
s+t& t& s& 1\\
-s+t& s& t& 1
\end{array}\right),
$$
and note that $M$ is not unimodular (setting $s=t=0$ makes the rank of $M$ drops). In contrast, we have that
$$\tilde{M}=\left(\begin{array}{ccccc}
s+t& t& s& 1& s+2t\\
-s+t& s& t& 1&t+2s+1
\end{array}\right)
$$
is unimodular, and by applying Quillen-Suslin to this matrix we obtain $U$  such that
$$\tilde M \cdot U =\begin{pmatrix}
1&0&0&0 &0\\
0&1&0 &0 &0
\end{pmatrix}.$$
In our case, we get
$$U=\left(
\begin{array}{ccccc}
 0 & 0 & 0 & 1 & 0 \\
 2 & -2 & 2 t-2 s & -4 s & 2 s-2 t+1 \\
 1 & -1 & -s+t+1 & -2 s & s-t \\
 1 & 0 & -s & -s-t & -t \\
 -1 & 1 & s-t & 2 s & t-s \\
\end{array}
\right),$$
and hence
$$\widehat{U}=\left(
\begin{array}{ccc}
 0 & 1 & 0 \\
 2 t-2 s & -4 s & 2 s-2 t+1 \\
 -s+t+1 & -2 s & s-t \\
 -s & -s-t & -t
\end{array}
\right)
$$
is a basis of $\mbox{Syz}(a_1, a_2, a_3, a_4).$ Note that the first columns of $U$ (except the last row) encode the matrix $N$, i.e. we can take this matrix as

$$N=
\left(
\begin{array}{cc}
 0 & 0 \\
 2 & -2 \\
 1 & -1 \\
 1 & 0
\end{array}
\right).$$
From here, we can extend it easily to a unimodular
$$N^*=\left(
\begin{array}{cccc}
 0 & 0 & 1 & 0 \\
 2 & -2 & 0 & 1 \\
 1 & -1 & 0 & 0 \\
 1 & 0 & 0 & 0 \\
\end{array}
\right),$$
and then we get
$$N^{**}=
\left(
\begin{array}{cccc}
 0 & 0 & 1 & 0 \\
 2 & -2 & -4 s & 2 s-2 t+1 \\
 1 & -1 & -2 s & s-t \\
 1 & 0 & -s-t & -t \\
\end{array}
\right).$$
To conclude, we need to replace the first two columns by a combination of them, to get
$$\widehat{N}=\left(
\begin{array}{ccc}
 0 & 1 & 0 \\
 -2 s+2 t-2 & -4 s & 2 s-2 t+1 \\
 -s+t-1 & -2 s & s-t \\
 s+2 t & -s-t & -t \\
\end{array}
\right)$$
which encodes another basis of $\mbox{Syz}(a_1, a_2, a_3, a_4).$ This basis is quite similar than the one we found via $U$. Indeed, the second and the third columns of both matrices coincide, while the first column of $\widehat{N}$ equals to the first minus 2 times the last  column of $\tilde{U}.$ This ``coincidence'' is explained by our choice of $N$ from the first two columns of the matrix $U$ above.

\bigskip
\section{Minimal $\mu$-bases of parametric surfaces}\label{final}
From a geometric point of view, the sequence $(a_1,\ldots, a_m)$ of polynomials in $\K[x_1,\ldots, x_n]$ can be regarded as the parametrization of a variety $Y\subset\K^m$ being the image of the map $(a_1,\ldots, a_m):\K^n\to\K^m.$  Understanding the role of $\mbox{Syz}(a_1,\ldots, a_m)$ in the study of geometric properties of $Y$ is an active area of research, and several results and challenges are well identified there, see for instance \cite{cox03}.  In particular, when $n=2,$ this map represents typically a surface in $\K^m$, and bounding the degrees of generators of the syzygy module has been posed as an open problem in \cite{CCL05} for the case $m=4,$ i.e. surfaces in $3$-dimensional space.

A first answer to this problem was posted in \cite{cid19}, where the author exhibits a general bound of order $\delta_a^{33}$ to solve this problem ($n=2,\,m=4$). In \cite{CDM20} we improved this bound to $\delta_a^{12}$ in the case the ideal $I$ has a so-called ``shape basis'', meaning that one can take $p$ and $q$ in \eqref{1} as $x_2-r$ and $s$ respectively, with $r, s\in\K[x_1].$  Indeed, our main result there (\cite[Theorem 1.2]{CDM20}) is that a basis of $\mbox{Syz}(a_1, a_2, a_3, a_4)$ in that case can be found with degree bounded by
\begin{equation}\label{bud}
5\delta_0^4(2\delta_a+1)^4.
\end{equation}

The approach to solve that problem differs significantly than the one presented here. For instance we only used there the ``simplified'' Effective Quillen-Suslin Theorem presented in \cite{FG90}, which is essentially Theorem \ref{eqs} in the case $r=1$. We only worked with the matrix which is called $M$ in \eqref{MyN}, and took advantage of the fact that one of its rows only depends on the variable $x_1$ in the shape-basis case. From there, after applying the Effective Quillen-Suslin algorithm to that univariate row $(r=1)$, a tricky manipulation of the remaining row would allow us to perform again the same simplified Effective Quillen-Suslin algorithm to the second row. The basis would come then after some simplifications and substitutions of these calculations.

In the present paper, we only apply once the Effective Quillen-Suslin algorithm described in \cite{CCDHKS93} with $r=2$ to the modified matrix $\tilde{M}$ instead of $M$. As a result, we obtain the matrix $\widehat{U}$ which encodes the elements of a basis of $\mbox{Syz}(a_1,\ldots, a_m)$ (Theorem \ref{mt}). In addition, we also get another algorithm over the matrix $N$ which gives another basis for this module (Theorem \ref{NN}). This approach was not considered in \cite{CDM20}.

When applied to the case $n=2,\,m=4,$ and using the fact that thanks to B\'ezout's Theorem one can take $\delta_0\leq\delta_a^2,$ the first bound in Theorem \ref{mtt} then amounts to a constant times $\delta_a^{16},$ while the second one is of the order of $\delta_a^8,$ which is better than the results obtained in \cite{CDM20} if one substitutes $\delta_0$ with $\delta_a^2.$

But we can actually improve the bound of $\delta_a^{16}$ from Theorem \ref{mtt} if we inspect carefully the structure of a matrix $M$ converting a shape basis into the input sequence $a_1\ldots a_4,$ as it was done in \cite{CDM20}. Indeed, we have

\begin{proposition}\label{citam}
If $a_1, a_2, a_3, a_4\in\K[s,t]$ have degrees bounded by $\delta_1$  and the ideal generated by them has a shape basis as in \cite{CDM20} of degree $\delta_0,$ a basis of $\mbox{Syz}(a_1, a_2, a_3, a_4)$ can be found with degree bounded by  $192(\delta_0\delta_a+1)^{4}.$
\end{proposition}
\begin{proof}
From the proof of Theorem 1.2 in \cite{CDM20}, we get that the matrix $M$ converting the shape basis into the original sequence has $\delta_M\leq \delta_a\delta_0.$
The result now follows by applying Theorem \ref{explicitqs} to $\tilde{M}$ with $n=r=2.$
\end{proof}

\begin{remark}
Note that the bound obtained in Proposition \ref{citam} is of the same order than the one in \eqref{bud} in terms of $\delta_0$ and $\delta_a.$
\end{remark}



\bigskip
\begin{thebibliography}{XXXXXX00}

\bibitem[BW89]{BW89}
Broida, Joel G.; Williamson, S. Gill.
\newblock{\em A comprehensive introduction to linear algebra.\/}
\newblock  Addison--Wesley Publishing Company, Advanced Book Program, Redwood City, CA, 1989.

\bibitem[CCDHKS93]{CCDHKS93}
Caniglia, Leandro; Corti\~nas, Guillermo; Dan\'on, Silvia; Heinz, Joos; Krick, Teresa; Solern\'o, Pablo.
\newblock{\em Algorithmic aspects of Suslin's proof of Serre's conjecture.\/}
\newblock Comput Complexity 3 (1993) 31--55.

\bibitem[CCL05]{CCL05}
Chen, Falai; Cox, David; Liu, Yang.
\newblock{\em The $\mu$-basis and implicitization of a rational parametric surface.\/}
\newblock J. Symbolic Comput. 39 (2005), no. 6, 689--706.


\bibitem[CDM20]{CDM20}
Cortadellas Ben\'itez, Teresa; D'Andrea, Carlos; Montoro, Eul\`alia.
\newblock{\em Bounds for degrees of minimal $\mu$-bases of parametric surfaces.\/}
\newblock Proc. ACM Intern. Symp. on Symbolic and Algebraic Computation (2020), 107--113.

\bibitem[Cid19]{cid19}
Cid-Ruiz, Yairon.
\newblock{\em Bounding the degrees of a minimal $\mu$-basis for a rational surface parametrization.\/}
\newblock J. Symbolic Comput. 95 (2019), 134--150.

\bibitem[Cox03]{cox03}
Cox, David.
\newblock{\em Curves, surfaces, and syzygies.\/}
\newblock Topics in algebraic geometry and geometric modeling, 131--150,
Contemp. Math., 334, Amer. Math. Soc., Providence, RI, 2003.

\bibitem[CSC98]{CSC98}
 Cox, David A.; Sederberg, Thomas W.; Chen, Falai.
 \newblock{\em The moving line ideal basis of planar rational curves.\/}
 \newblock Comput. Aided Geom. Design 15 (1998), no. 8, 803--827.


\bibitem[CW03]{CW03}
Chen, Falai; Wang, Wenping.
\newblock{\em Revisiting the $\mu$-basis of a rational ruled surface.\/}
\newblock J. Symbolic Comput. 36 (2003), no. 5, 699--716.

\bibitem[DCS05]{DCS05}
Deng, Jiansong; Chen, Falai; Shen, Liyong.
\newblock{\em Computing $\mu$-bases of rational curves and surfaces using polynomial matrix factorization.\/}
\newblock ISSAC'05, 132--139, ACM, New York, 2005.

\bibitem[DFGS91]{DFGS91}
 Dickenstein, Alicia; Fitchas, Noa\"i; Giusti, Marc; Sessa, Carmen.
 \newblock{\em The membership problem for unmixed polynomial ideals is solvable in single exponential time.\/}
 \newblock Applied algebra, algebraic algorithms, and error-correcting codes (Toulouse, 1989). Discrete Appl. Math. 33 (1991), no. 1--3, 73--94.

\bibitem[Eis05]{eis05}
Eisenbud, David. 2005.
\newblock{\em The geometry of syzygies.\/}
\newblock  Graduate Texts in Mathematics, Vol. 229. Springer-Verlag, New York.
\newblock A second course in commutative algebra and algebraic geometry.

\bibitem[FG90]{FG90}
Fitchas, Noa\"i; Galligo, Andr\'e.
\newblock{\em Nullstellensatz effectif et conjecture de Serre (th\'eor\`eme de Quillen-Suslin) pour le calcul formel.\/}
\newblock Math. Nachr. 149 (1990), 231--253.

\bibitem[GB82]{GB82}
Guiver, John P.; Bose, N. K.
\newblock{\em Polynomial matrix primitive factorization over arbitrary coefficient field and related results.\/}
\newblock IEEE Trans. Circuits and Systems 29 (1982), no. 10, 649--657.


\bibitem[Has09]{has09}
Hashemi, Amir.
\newblock{\em Nullstellens\"atze for zero-dimensional Gr\"obner bases.\/}
\newblock Comput. Complexity 18 (2009), no. 1, 155--168.

\bibitem[HHK17]{HHK17}
Hong, Hoon; Hough, Zachary; Kogan, Irina A.
\newblock{\em Algorithm for computing $\mu$-bases of univariate polynomials.\/}
\newblock J. Symbolic Comput. 80 (2017), part 3, 844--874.

\bibitem[Jel05]{jel05}
Jelonek, Zbigniew.
\newblock{\em On the effective Nullstellensatz.\/}
\newblock Invent. Math. 162 (2005), no. 1, 1--17.

\bibitem[Map20]{map20}
 {\em  Maple 2020.1.}
 \newblock Maplesoft, a division of Waterloo Maple Inc., Waterloo, Ontario.

\bibitem[Math18]{math18}
Wolfram Research, Inc. 2018.
\newblock{\em Mathematica.\/}
\newblock Version 11. (2018). Champaign, IL.

\bibitem[SL17]{SL17}
Shen, Li-Yong; Goldman, Ron.
\newblock{\em Algorithms for computing strong ?-bases for rational tensor product surfaces.\/}
\newblock Comput. Aided Geom. Design 52/53 (2017), 48--62.


\bibitem[Wal78]{wal78} Walker, Robert J.
\newblock{\em Algebraic curves.\/}
\newblock Reprint of the 1950 edition. Springer--Verlag, New York-Heidelberg, 1978.

\bibitem[YFJS19]{YFJS19}
Yao, Shanshan; Feng, Yifei; Jia, Xiaohong; Shen, Li-Yong.
\newblock{\em A package to compute implicit equations for rational curves and surfaces.\/}
\newblock ACM Commun. Comput. Algebra 53 (2019), no. 2, 33--36.

\bibitem[YJ19]{YJ19}
Yao, Shanshan; Jia, Xiaohong.
\newblock{\em  $\mu$-bases for rational canal surfaces.\/}
\newblock  Comput. Aided Geom. Design 69 (2019), 11--26.


 \end{thebibliography}


\end{document}